\documentclass[a4paper,12pt]{amsart}
\usepackage[utf8]{inputenc}
\usepackage{hyperref}
\usepackage{amsmath, amssymb}
\usepackage[alphabetic,nobysame, initials]{amsrefs}
\usepackage{graphicx}

\newtheorem{thm}{Theorem}[section]
\newtheorem{lem}[thm]{Lemma}
\newtheorem{cor}[thm]{Corollary}
\theoremstyle{definition}
\newtheorem{defn}[thm]{Definition}
\newtheorem{obs}[thm]{Observation}

\renewcommand{\Re}{\mathbb R}
\renewcommand{\epsilon}{\varepsilon}
\newcommand{\Ren}{\Re^n}
\newcommand{\Renp}{\Re^{n+1}}
\newcommand{\Sonp}{SO(n+1)}
\newcommand{\Se}{\mathbb S}
\newcommand{\Sen}{\Se^n}
\newcommand{\B}{\mathbf B}

\newcommand{\FF}{\mathcal F}

\newcommand{\HH}{\mathcal H}

\newcommand{\st}{\; : \; }
\renewcommand{\phi}{\varphi}
\DeclareMathOperator{\inter}{int}
\DeclareMathOperator{\cl}{cl}

\DeclareMathOperator{\vol}{vol}

\DeclareMathOperator{\card}{card}

\newcommand\fact[2][]{
  \ifx&#1& 
  \refstepcounter{equation}
  \fi
  \begin{minipage}{0.09\textwidth}
  \ifx&#1& 
  (\theequation)
  \else 
  {#1}
  \fi
  \end{minipage}
  \hfill
  \begin{minipage}{0.81\textwidth}
  \emph{
  \begin{sloppypar}
  #2
  \end{sloppypar}
  }
  \end{minipage}
  \smallskip
}

\title{On some covering problems in geometry}
\author[M. Nasz\'odi]{M\'arton Nasz\'odi}
\address{
M\'arton Nasz\'odi,
Dept. of Geometry,
Lorand E\"otv\"os University,
P\'azm\'any P\'eter S\'et\'any 1/C
Budapest, Hungary 1117
}
\email{marton.naszodi@math.elte.hu}
\thanks{The author acknowledges the support of the J\'anos Bolyai 
Research Scholarship of the Hungarian Academy of Sciences, and
the Hung. Nat. Sci. Found. (OTKA) grant PD104744.}

\begin{document}
\begin{abstract}
We present a method to obtain upper bounds on covering numbers.
As applications of this method, we reprove and generalize results of Rogers 
on economically covering Euclidean $n$-space with translates 
of a convex body, or more generally, any measurable set.
We obtain a bound for the density of covering the 
$n$-sphere by rotated copies of a spherically convex set (or, any measurable 
set). Using the same method, we sharpen an estimate by 
Artstein--Avidan and Slomka on covering a bounded set by translates of another.

The main novelty of our method is that it is not probabilistic. The key idea,
which makes our proofs rather simple and uniform through different settings, is 
an algorithmic result of Lovász and Stein.
\end{abstract}

\maketitle

\section{Introduction}

Given two sets $K$ and $L$ in $\Ren$ (resp. $\Sen$), and we want to cover $K$ 
by as few translates (resp. rotated copies) of $L$ as possible. Upper bounds 
for these kind of covering problems are often obtained by probabilistic 
methods, that is, by taking randomly chosen copies of $L$. We present a method 
that relies on an algorithmic result of Lov\'asz and Stein, and yields proofs 
that are 
simple, non-probabilistic and quite uniform through different geometric 
settings.

For two Borel measurable sets $K$ and $L$ in $\Ren$, let $N(K,L)$ denote the 
\emph{translative covering number} of $K$ by $L$ ie. the minimum number of 
translates of $L$ that cover $K$.

\begin{defn}\label{defn:fraccovcvx}
 Let $K$ and $L$ be bounded Borel measurable sets in $\Ren$. A \emph{fractional 
covering} of $K$ by translates of $L$ is a Borel measure $\mu$ on $\Ren$ with 
$\mu(x-L)\geq 1$ for all $x\in K$. The \emph{fractional covering number} of $K$ 
by translates of $L$ is
\[
 N^\ast(K,L)=
 \]\[\inf\left\{\mu(\Ren)\st \mu \mbox{ is a fractional covering of } K 
\mbox{ by translates of } L\right\}.
\]
\end{defn}

Clearly, in Definition~\ref{defn:fraccovgen} we may assume that a fractional 
cover $\mu$ is supported on $\cl(K-L)$. 
According to Theorem~1.7 of \cite{AS}), we have
\begin{equation}\label{eq:simpleupperbound}
\max\left\{\frac{\vol(K)}{\vol(L)},1\right\}\leq 
N^\ast(K,L)\leq\frac{\vol(K-L)}{\vol(L)}.
\end{equation}

Here, the upper bound is easy to see, as the Lebesgue measure restricted to 
$K-L$ with the following scaling $\mu=\vol/\vol(L)$ is a fractional cover of 
$K$ by translates of $L$.

For two sets $K, T\subset\Ren$, we denote their \emph{Minkowski difference} by 
$K\sim T=\{x\in \Ren\st T+x \subseteq K\}$. 

\begin{thm}\label{thm:cvxIG}
  Let $K, L$ and $T$ be bounded Borel measurable sets in $\Ren$ and let 
$\Lambda\subset\Ren$ be a finite set with $K\subseteq \Lambda+T$. Then
  \begin{equation}\label{eq:cvxIG}
  N(K,L)\leq 
  \end{equation}
\[ 
  (1+\ln(
  \max_{x\in K-L} \card ((x+(L\sim T))\cap \Lambda )
  ) )
  \cdot N^\ast(K-T,L\sim T).
\]
  If $\Lambda\subset K$, then we have
  \begin{equation}\label{eq:cvxIGspec}
  N(K,L)\leq 
  \end{equation}
\[  (1+\ln(
  \max_{x\in K-L} \card ((x+(L\sim T))\cap \Lambda )
  ) )
  \cdot N^\ast(K,L\sim T).
\]
\end{thm}

For a set $K\subset\Ren$ and $\delta>0$, we denote the \emph{$\delta$-inner 
parallel body} of $K$ by $K_{-\delta}:=K\sim \B(o,\delta)=\{x\in K\st 
\B(x,\delta)\subseteq K\}$, where $\B(x,\delta)$ denotes the Euclidean ball of 
radius $\delta$ centered at $x$. As an application of Theorem~\ref{thm:cvxIG}, 
we will obtain

\begin{thm}\label{thm:Renbyanything}
Let $K\subseteq\Ren$ be a bounded measurable set.
Then there is a covering of $\Ren$ by translated copies of $K$ of density at 
most
\[
 \inf_{\delta>0}\left[
  \frac{\vol(K)}{\vol(K_{-\delta})}
  \left( 
1+\ln\frac{\vol\left(K_{-\delta/2}\right)}{\vol\left(\B\left(o,\frac{\delta}{2}
\right)\right)
}\right)\right].
\]
\end{thm}
The $\delta$-inner parallel body could be defined with respect to a norm that 
is distinct from the Euclidean. As is easily seen from the proof, the 
theorem would still hold.

Now, we turn to coverings on the sphere.
We denote the Haar probability measure on $\Sen\subset\Renp$ by $\sigma$, the 
closed spherical cap of spherical radius $\phi$ centered at $u\in\Sen$ by 
$C(u,\phi)$, and its measure by $\Omega(\phi)=\sigma(C(u,\phi))$. For a set 
$K\subset\Sen$ and $\delta>0$, we denote the \emph{$\delta$--inner parallel 
body} of $K$ by $K_{-\delta}=\{u\in K\st C(u,\delta)\subseteq K\}$.

A set $K\subset\Sen$ is called \emph{spherically convex}, if it is contained in 
an open hemisphere and for any two of its points, it contains the shorter great 
circular arc connecting them.

The \emph{spherical circumradius} of a subset of an open hemisphere of $\Sen$ 
is the spherical radius of the smallest spherical cap (the \emph{circum-cap}) 
that contains the set.

\begin{thm}\label{thm:spherebyanything}
Let $K\subseteq\Sen$ be a measurable set.
Then there is a covering of $\Sen$ by rotated copies of $K$ of density at most
\[
 \inf_{\delta>0}\left[
  \frac{\sigma(K)}{\sigma(K_{-\delta})}
  \left( 
1+\ln\frac{\sigma\left(K_{-\delta/2}\right)}{\Omega\left(\frac{\delta}{2}\right)
}\right)\right].
\]
\end{thm}

\begin{cor}\label{cor:spherebyconvex}
Let $K\subseteq\Sen$ be a spherically convex set of spherical circumradius 
$\rho$.
Then there is a covering of $\Sen$ by rotated copies of $K$ of density at most
\[
 \inf_{\kappa>0\st K_{-\left(\kappa\rho\right)}\neq\emptyset}\left[ 
  \frac{\sigma(K)}{\sigma(K)-\Omega(\rho)\left(1-(1-\kappa)^n\right)}
  \left( 
2n+n\ln\frac{1}{\kappa\rho}
  \right)\right].
\]
\end{cor}

We prove the Euclidean results in Section~\ref{sec:Renbound}, and the spherical 
results in Section~\ref{sec:Sphbound}.

\section{History}\label{sec:history}

An important point in the theory of coverings in geometry is the following 
theorem of Rogers \cite{Ro57}. For a definition of the covering density, cf. 
\cite{RoBook64}.
\begin{thm}[Rogers, \cite{Ro57}]\label{thm:Rogers}
  Let $K$ be a bounded convex set in $\Ren$ with non-empty interior. Then the 
covering density of $K$ is at most
  \begin{equation}
  \theta(K)\leq n\ln n+ n\ln\ln n + 5n.
  \end{equation}
\end{thm}

Earlier, exponential upper bounds for the covering density were obtained by 
Rogers, Bambah and Roth, and for the special case of the Euclidean ball by 
Davenport and Watson (cf. \cite{Ro57} for references).
The current best bound is due to G. Fejes Tóth \cite{FTG09}, who replaced the 
last term $5n$ by $n+o(n)$.

We will obtain Theorem~\ref{thm:Rogers} as a corollary to our more general 
Theorem~\ref{thm:Renbyanything}.

Another classical example of a geometric covering problem is the following.
Estimate the minimum number of spherical caps of radius $\phi$ needed 
to cover the sphere $\Sen$ in $\Renp$.
\begin{thm}[B\"or\"oczky Jr. and Wintsche, \cite{BW03}]\label{thm:spherebycaps}
Let $0<\phi<\frac{\pi}{2}$. Then there is a covering of $\Sen$ by spherical 
caps of radius $\phi$ with density at most $n\ln n+n\ln\ln n+5n$. 
\end{thm}
This estimate was proved in \cite{BW03} improving an earlier result of Rogers 
\cite{Ro63}. The current best bound is better when $\phi<\frac{\pi}{3}$: Dumer 
\cite{Du07} gave a covering in this case of density at most $\frac{n\ln n}{2}$.

We will obtain Theorem~\ref{thm:spherebycaps} as a corollary to our more 
general Theorem~\ref{thm:spherebyanything}.

The fractional version of $N(K,\inter K)$ (see 
Definition~\ref{defn:fraccovgen}) first appeared in \cite{Na09} and in general 
for $N(K,L)$ in \cite{AR11} and \cite{AS}.

A result very similar to our Theorem~\ref{thm:cvxIG} appeared as 
Theorem~1.6 in the paper \cite{AS} by Artstein-Avidan and Slomka. The main 
differences are the following. 
Quantitatively, our result is somewhat stronger by having $\max \card 
(\dots)$ in the logarithm as opposed to $\card \Lambda$. 
This allows us to obtain Theorems~\ref{thm:Rogers} and \ref{thm:Renbyanything} 
as corollaries to Theorem~\ref{thm:cvxIG}.
Furthermore, we have no minor term of order $\sqrt{\ln (\card 
\Lambda)(N^\ast+1)}$.
The method of the proof in \cite{AS} consist of two parts. One is to reduce the 
problem to a finite covering problem by replacing $K$ by a sufficiently dense 
finite set (a $\delta$-net). Next, a probabilistic argument is used to solve 
the finite covering problem.
A similar route is followed in \cite{FuKa08} where a variant of 
Theorem~\ref{thm:Rogers} (previously obtained in \cite{ErRo61}) is proved 
(using Lov\'asz's Local Lemma) according to which such low density covering of 
$\Ren$ by translates of $K$ exists with the additional requirement that no 
point is covered too many times.
An even earlier appearance of this method in the context of the illumination 
problem can be found in \cite{Sch88}.
A major contribution of \cite{AS} is that they used this method to bridge the 
gap between $N$ and $N^\ast$, that is, they noticed that the method works with 
any measure, not just the volume.

We also use the first part of the method (taking a $\delta$-net), but then 
replace the second (probabilistic) part by a simple application of 
a non-probabilistic result, Lemma~\ref{lem:Lovasz}.

\section{Preliminaries}\label{sec:prelim}

We start with introducing some combinatorial notions.
\begin{defn}\label{defn:fraccovgen}
 Let $Y$ be a set, $\FF$ a family of subsets of $Y$ and $X\subseteq Y$. A 
\emph{covering} of $X$ by $\FF$ is a subset of $\FF$ whose union contains $X$. 
The \emph{covering number} $\tau(X,\FF)$ of $X$ by $\FF$ is the 
minimum cardinality of its coverings by $\FF$.

A \emph{fractional covering} of $X$ by $\FF$ is a measure $\mu$ on $\FF$ with 
\[
\mu(\{F\in\FF\st x\in F\})\geq 1\;\;\;\mbox{ for all } x\in X.
\]
The \emph{fractional covering number} of $\FF$ is
\[
 \tau^\ast(X,\FF)=\inf\left\{\mu(\FF)\st \mu \mbox{ is a fractional covering of 
} X \mbox{ by } \FF\right\}.
\]
When a group $G$ acts on $Y$ and $\FF$ is the set $\{g(A)\st g\in G\}$ for some 
fixed subset $A$ of $Y$, we will identify $F\in\FF$ with $\{g\in G\st 
g(A)=F\}\subseteq G$ and thus, we will call a measure $\mu$  on $G$ a 
fractional covering of $X$ by $G$ if 
\[
\mu(\{g\in G \st x\in g(A)\})\geq 1\;\;\;\mbox{ for all } x\in X.
\]
\end{defn}
For more on (fractional) coverings, cf. \cite{Fu88} in the abstract 
(combinatorial) setting and \cite{Ma02} in the geometric setting.

The gap between $\tau$ and $\tau^\ast$ is bounded in the case of finite set 
families (hypergraphs) by the following result of Lovász \cite{Lo75} and 
Stein\cite{St74}.

\begin{lem}[Lovász \cite{Lo75}, Stein\cite{St74}]\label{lem:Lovasz}
 For any finite $\Lambda$ and $\HH\subseteq 2^\Lambda$ we have
  \begin{equation}\label{eq:LovaszIG}
    \tau(\Lambda,\HH) < (1+\ln(\max_{H\in \HH}\card H))\tau^\ast(\Lambda,\HH).
  \end{equation}
  Furthermore, the greedy algorithm (always picking the set that covers the 
most number of uncovered points) yields a covering of cardinality less than the 
right hand side in \eqref{eq:LovaszIG}.
\end{lem}

The following straightforward corollary to Lemma~\ref{lem:Lovasz} is a key 
element of our proofs.
\begin{obs}\label{obs:IG}
Let $Y$ be a set, $\FF$ a family of subsets of $Y$, and $X\subseteq Y$. 
Let $\Lambda$ be a finite subset of $Y$ and $\Lambda\subseteq U\subseteq Y$. 
Assume 
that for another family $\FF^\prime$ of subsets of $Y$ we have 
$\tau(X,\FF)\leq\tau(\Lambda,\FF^\prime)$. Then
  \begin{equation}\label{eq:IG}
  \tau(X,\FF)\leq 
  \tau(\Lambda,\FF^\prime)\leq
  (1+\ln( \max_{F^\prime\in\FF^\prime} \card \{\Lambda\cap F^\prime \} ) ) 
  \cdot \tau^\ast(U, \FF^\prime).
  \end{equation}
\end{obs}

We will rely on the following estimates of $\Omega$ by B\"or\"oczky and 
Wintsche \cite{BW03}.

\begin{lem}[{B\"or\"oczky -- Wintsche \cite{BW03}}]\label{lem:BWcapsize} Let 
$0<\phi<\pi/2$.
\begin{eqnarray}
 \Omega(\phi)&>& \frac{\sin^n\phi}{\sqrt{2\pi(n+1)}}, 
\label{eq:BWnagy}
 \\
 \Omega(\phi)&<& \frac{\sin^n\phi}{\sqrt{2\pi n}\cos\phi} 
,\;\;\;\;\mbox{ if } \phi\leq \arccos \frac{1}{\sqrt{n+1}},
\label{eq:BWkicsi}
\\
 \Omega(t\phi)&<& t^n\Omega(\phi),\;\;\;\;\mbox{ if } 1<t<\frac{\pi}{2\phi}. 
\label{eq:BWtszer}
\end{eqnarray}
\end{lem}

The following is known as Jordan's inequality:
\begin{equation}\label{eq:jordan}
\frac{2x}{\pi}\leq\sin x\;\;\mbox{ for }\;\;x\in[0,\pi/2]
\end{equation}

\section{Proof of the covering results in 
\texorpdfstring{$\Ren$}{Rn}}\label{sec:Renbound}

We present these proofs in the order of their difficulty. In this way, ideas 
and 
technicalities are --perhaps-- easier to separate.

\begin{proof}[Proof of Theorem~\ref{thm:cvxIG}]
The proof is simply a substitution into \eqref{eq:IG}. We take $Y=\Ren$, $X=K$, 
$\FF=\{L+x\st x\in K-L\}$, $\FF^\prime=\{L\sim T+x\st x\in K-L \}$. One can 
take $U=K-T$ as any member of $\Lambda$ not in $K-T$ could be 
dropped from $\Lambda$ and $\Lambda$ would still have the property that 
$\Lambda 
+T \supseteq K$. That proves \eqref{eq:cvxIG}. To prove 
\eqref{eq:cvxIGspec}, we notice that in the case when $\Lambda\subset K$, 
one can take $U=K$.
\end{proof}

\begin{proof}[Proof of Theorem~\ref{thm:Renbyanything}]
 Let $C$ denote the cube $C=[-a,a]^n$, where $a>0$ is large. 
Our goal is to cover $C$ by translates of $K$ economically.

Fix $\delta>0$, and let $\Lambda\subset\Ren$ be a 
finite set such that $\Lambda+\B(o,\delta/2)$ is a saturated (ie. maximal) 
packing of $\B(o,\delta/2)$ in $C+\B(o,\delta/2)$. Thus, by the 
maximality, we have that $\Lambda$ is a $\delta$-net of $C$ with respect to 
the Euclidean distance, ie. $\Lambda+\B(o,\delta)\supseteq C$.

By considering volume, for any $x\in \Ren$ we have
\begin{equation}\label{eq:lambdasmallRR}
 \card\big({\Lambda\cap (x+K_{-\delta})}\big)\leq 
 \frac{\vol \left(K_{-\delta} 
+\B(o,\delta/2)\right)}{\vol\left(\B(o,\delta/2)\right)}\leq
 \frac{\vol \left(K_{-\delta/2}\right)}{\vol\left(\B(o,\delta/2)\right)}.
\end{equation}

Let $\varepsilon>0$ be fixed. Clearly, if $a$ is sufficiently large then
\begin{equation}\label{eq:nstarobviousRR}
 N^\ast(C+\B(o,\delta/2),K_{-\delta})\leq
\frac{\vol\left(C+\B(o,\delta/2)-K_{-\delta}\right)}{\vol K_{-\delta}}\leq
 (1+\epsilon)\frac{\vol C}{\vol K_{-\delta}}.
\end{equation}

By \eqref{eq:cvxIG}, \eqref{eq:lambdasmallRR} and \eqref{eq:nstarobviousRR} we 
have
\begin{equation*}
 N(C,K)\leq 
  (1+\epsilon)
  \left(1+\ln
   \frac{\vol \left(K_{-\delta/2}\right)}{\vol\left(\B(o,\delta/2)\right)}
  \right)
  \frac{\vol C}{\vol K_{-\delta}}.
\end{equation*}

Finally,
\begin{equation*}
\label{eq:thetasymmRR}
 \theta(K)\leq 
 N(C,K)\vol(K)/\vol(C)
\end{equation*}
yields the promised bound.
\end{proof}

\begin{proof}[Proof of Theorem~\ref{thm:Rogers}]
 Let $C$ denote the cube $C=[-a,a]^n$, where $a>0$ is large. 
Our goal is to cover $C$ by translates of $K$ economically.
First, consider the case when $K=-K$. 

Let $\delta>0$ be fixed (to be chosen later) and let $\Lambda\subset\Ren$ be a 
finite set such that $\Lambda+\frac{\delta}{2}K$ is a saturated (ie. maximal) 
packing of $\frac{\delta}{2}K$ in $C-\frac{\delta}{2}K$. Thus, by the 
maximality, we have that $\Lambda$ is a $\delta$-net of $C$ with respect to 
$K$, ie.
$\Lambda+\delta K\supseteq C$.
By considering volume, for any $x\in \Ren$ we have
\begin{equation}\label{eq:lambdasmall}
 \card\big({\Lambda\cap (x+(1-\delta)K)}\big)\leq 
 \frac{\vol \left((1-\delta)K 
+\frac{\delta}{2}K\right)}{\vol\left(\frac{\delta}{2}K\right)}\leq
 \left(\frac{2}{\delta}\right)^n.
\end{equation}

Let $\varepsilon>0$ be fixed. Clearly, if $a$ is sufficiently large then
\begin{equation}\label{eq:nstarobvious}
 N^\ast(C-\delta K,(1-\delta)K)\leq
 (1+\epsilon)\frac{\vol C}{(1-\delta)^n\vol K}.
\end{equation}

By \eqref{eq:cvxIG}, \eqref{eq:lambdasmall} and \eqref{eq:nstarobvious} we have
\begin{equation*}
 N(C,K)\leq 
  \frac{1+\epsilon}{(1-\delta)^n}
  \left(1+n\ln
  \left(\frac{2}{\delta}\right)\right)
  \frac{\vol C}{\vol K}.
\end{equation*}

On the other hand,
\begin{equation}
\label{eq:thetasymm}
 \theta(K)\leq 
 N(C,K)\vol(K)/\vol(C)\leq
  \frac{1+\epsilon}{(1-\delta)^n}
  \left(1+n\ln
  \left(\frac{2}{\delta}\right)\right).
\end{equation}


We choose $\delta=\frac{1}{2n\ln n}$, and the following standard computation 
\begin{eqnarray}\label{eq:lnnesszam}
  (1+\varepsilon)^{-1}\theta(K)\leq
  \left(1+n\ln (4n\ln n)\right)\exp(1/\ln n)
  \\
  \leq
  \left(1+n\ln (4n\ln n)\right)  
   (1+2/\ln n)
  \leq
  \left(n\ln n+n\ln\ln n+ 5n\right),\nonumber
\end{eqnarray}
yields the desired bound (as $\varepsilon$ can be taken arbitrarily close to 0).

Next, consider the general case, that is when $K$ is not necessarily symmetric 
about the origin. We need to make the following mo\-di\-fi\-ca\-tions. 
Milman and Pajor (cf. Corollary~3 of \cite{MiPa00}) showed that, if the 
centroid (that is, the center of mass) of $K$ is the origin, then 
$\vol(K\cap-K)\geq \frac{\vol K}{2^n}$. (Note that the existence of a translate 
of $K$ for which this inequality holds was proved by Stein \cite{St56} using a 
probabilistic argument.) We define $\Lambda$ as a saturated 
packing of translates of $\frac{\delta}{2}(K\cap -K)$ in 
$C-\frac{\delta}{2}(K\cap -K)$. Thus, we have $C\subseteq \Lambda+\delta 
(K\cap-K)\subseteq \Lambda+\delta K$. Instead of \eqref{eq:lambdasmall}, we now 
have
\begin{equation*}
 \card\big({\Lambda\cap (x+(1-\delta)K)}\big)\leq 
 \left(\frac{4}{\delta}\right)^n.
\end{equation*}
for any $x\in\Ren$. Rolling this change through the proof, at the end in place 
of \eqref{eq:thetasymm}, we obtain
\begin{equation*}
 \theta(K)\leq 
\frac{1+\epsilon}{(1-\delta)^n}
  \left(1+n\ln
  \left(\frac{4}{\delta}\right)\right),
\end{equation*}
which, however, is still less than
$(1+\epsilon)\left(n\ln n+n\ln\ln n+ 5n\right)$ with the same choice of 
$\delta=\frac{1}{2n\ln n}$.
\end{proof}

\section{Proof of the spherical results}\label{sec:Sphbound}

\begin{proof}[Proof of Theorem~\ref{thm:spherebyanything}]
Let $\Lambda$ be the set of centers of a saturated (ie. maximal) packing of 
caps of radius $\delta/2$.
Clearly, $\Lambda$ is a $\delta$-net of $\Sen$, and thus, if we cover 
$\Lambda$ by rotated copies of radius $K_{-\delta}$, then the same rotations 
yield a covering of $\Sen$ by copies of $K$.

Let $\bar\sigma$ denote the probability Haar measure on $\Sonp$.
Let $H\subset\Sen$ be a measurable set, 
and denote the family of rotated copies of $H$ by
$\FF(H)=\{AH\st A\in\Sonp\}$.
Recall that for any fixed $u\in\Sen$ we have
\begin{eqnarray*}
 \bar\sigma(\{A\in\Sonp\st u\in AH\})
 =\\
 \bar\sigma(\{A\in\Sonp\st u\in A^{-1}H\})=
\nonumber
 \\
 \bar\sigma(\{A\in\Sonp\st Au\in H\})
 =
\sigma(H) \nonumber.
\end{eqnarray*}
It follows that the measure $\frac{\bar\sigma}{\sigma(H)}$ on $\Sonp$ is a 
fractional cover of $\Sen$ by $\FF(H)$ and thus,
$\tau^{\ast}(\Sen,\FF(H))\leq\frac{1}{\sigma(H)}$.

Thus by \eqref{eq:IG}, we obtain the following for the density of a covering by 
rotated images of $K$:
\[
\mbox{ density } \leq
 \sigma(K)\tau(\Sen,\FF(K))\leq\sigma(K)\tau(\Lambda, \FF(K_{-\delta}))
\]
\[
  \leq
  (1+\ln( \max_{A\in\Sonp} \card \{\Lambda\cap AK_{-\delta} \} ) ) 
  \cdot \frac{\sigma(K)}{\sigma(K_{-\delta})}
\]
\[
 \leq
  \frac{\sigma(K)}{\sigma(K_{-\delta})}
  \left( 
1+\ln\frac{\sigma\left(K_{-\delta/2}\right)}{\Omega\left(\frac{\delta}{2}\right)
}\right). 
\]
Since it holds for any $\delta>0$, the theorem follows.
\end{proof}

\begin{proof}[Proof of Theorem~\ref{thm:spherebycaps}]
We will apply Theorem~\ref{thm:spherebyanything} with $K$ being a cap of 
spherical radius $\phi$.
We set $\delta=\eta\phi$, where $\eta$ will be specified later.
By Theorem~\ref{thm:spherebyanything} and \eqref{eq:BWtszer}, we obtain for the 
density of a covering of $\Sen$ by caps of radius $\phi$:
\begin{equation*}
\mbox{ density } \leq
  \left(1+n\ln\left(\frac{2}{\eta}\right) \right) 
  \cdot \left(\frac{1}{1-\eta}\right)^n.
\end{equation*}
We choose $\eta=\frac{1}{2n\ln n}$, and the same computation as in 
\eqref{eq:lnnesszam} yields the desired bound.
\end{proof}

\begin{proof}[Proof of Corollary~\ref{cor:spherebyconvex}]
We set $\delta=\kappa\rho$.
First, observe that the measure of the belt-like region ${K\setminus 
K_{-\delta}}$ at the boundary of $K$ is at most as large as the measure of the 
belt-like region $C(v,\rho)\setminus C(c,\rho-\delta)$ at the boundary of the 
circum-cap $C(v,\rho)$ of $K$.

Next, combine
$\ln\frac{\sigma\left(K_{-\delta/2}\right)}{\Omega\left(\frac{\delta}{2}\right)}
\leq
\ln\frac{1}{\Omega\left(\frac{\delta}{2}\right)}$ with \eqref{eq:BWtszer} and 
\eqref{eq:jordan} to obtain the statement.
\end{proof}

\section*{Acknowledgement}
The author is grateful for the conversations with Károly Bezdek, Gábor Fejes 
Tóth and J\'anos Pach.

\bibliographystyle{alpha}
\bibliography{biblio}

\end{document}